\documentclass{amsart}

\newcommand{\cD}{{\mathcal D}}

\newcommand{\bN}{{\mathbb N}}
\newcommand{\bZ}{{\mathbb Z}}
\newcommand{\bQ}{{\mathbb Q}}
\newcommand{\bR}{{\mathbb R}}
\newcommand{\bC}{{\mathbb C}}

\newcommand{\be}{\begin{equation}}
\newcommand{\ee}{\end{equation}}
\newcommand{\ba}{\begin{eqnarray}}
\newcommand{\ea}{\end{eqnarray}}
\newcommand{\ban}{\begin{eqnarray*}}
\newcommand{\ean}{\end{eqnarray*}}

\newtheorem{theorem}{Theorem}[section]

\theoremstyle{definition}
\newtheorem{definition}[theorem]{Definition}
\newtheorem{example}[theorem]{Example}

\theoremstyle{remark}

\theoremstyle{proposition}
\newtheorem{proposition}[theorem]{Proposition}

\theoremstyle{proposition}
\newtheorem{corollary}[theorem]{Corollary}

\numberwithin{equation}{section}

\begin{document}

\title{$p$-Adic refinable functions and MRA-based wavelets}

\author{A.~Yu.~Khrennikov}
\address{International Center for Mathematical Modelling in Physics and
Cognitive Sciences MSI, V\"axj\"o University, \ SE-351 95,
V\"axj\"o, \ Sweden.}
\email{andrei.khrennikov@msi.vxu.se}
\thanks{This paper was supported in part by the grant
of The Swedish Royal Academy of Sciences on collaboration with
scientists of former Soviet Union.
The second author (V.~S.) was supported in part by
DFG Project 436 RUS 113/809 and Grant 05-01-04002-NNIOa of
RFBR.
The third author (M.S.) was supported in part by
Grant 06-01-00457 of RFBR}

\author{V.~M.~Shelkovich}
\address{Department of Mathematics, St.-Petersburg State Architecture
and Civil Engineering University, \ 2 Krasnoarmeiskaya 4, 190005,
St. Petersburg, \ Russia.}
\email{shelkv@vs1567.spb.edu}

\author{M.~Skopina}
\address{Department of Applied Mathematics and Control Processes,
St. Petersburg State University, \ Universitetskii pr.-35,
198504 St. Petersburg, Russia.}
\email{skopina@MS1167.spb.edu}





\begin{abstract}
We described a wide class of $p$-adic refinable
equations generating $p$-adic multiresolution analysis. A method
for the construction of $p$-adic orthogonal wavelet bases within
the framework of the MRA theory is suggested. A realization of this
method is illustrated by an example, which gives a new 3-adic wavelet
basis. Another realization leads to the $p$-adic Haar bases which were
known before.
\end{abstract}

\maketitle

\section{Introduction}
\label{s1}

A sensation happened in the early nineties because
a general scheme for the construction of wavelets was developed. This
scheme is based on the notion of multiresolution analysis (MRA in the sequel)
introduced by Y.~Meyer and S.~Mallat. Immediately specialists started to
implement new wavelet systems.
Nowadays it is difficult to find an engineering area where wavelets are
not applied. In the $p$-adic setting,  situation is the following.
In 2002 S.~V.~Kozyrev~\cite{Koz0} found a compactly supported
$p$-adic wavelet basis for $L^2(\bQ_p)$ which is an analog of the Haar basis.
Another $p$-adic wavelet-type system generalizing Kozyrev's basis was constructed
in~\cite{Kh-Sh1},~\cite{Kh-Sh2}.

It turned out that the mentioned above $p$-adic wavelets are eigenfunctions of
$p$-adic pseudo-differential operators~\cite{Al-Kh-Sh3},~\cite{Al-Kh-Sh5},
~\cite{Kh-Koz2}--~\cite{Kh-Sh2},~\cite{Koz0},~\cite{Koz2}.
This fact implies that study of wavelets is important and gives a new powerful
technique for solving $p$-adic problems (areas of applications can be found
in~\cite{Kh2},~\cite{Koch3},~\cite{Vl-V-Z}).

Nevertheless, in the cited papers, {\em a theory describing common
properties of $p$-adic wavelet bases and giving general methods for their
construction was not developed}. Moreover, it was assumed in~\cite{Ben-Ben}
that there are obstacles for creating a $p$-adic analog of the classical MRA theory.
To construct a $p$-adic analog of a classical MRA we need a proper
$p$-adic {\em refinement equation}. In~\cite{Kh-Sh1} the following conjecture
was proposed: the equality
\begin{equation}
\label{62.0-3}
\phi(x)=\sum_{r=0}^{p-1}\phi\Big(\frac{1}{p}x-\frac{r}{p}\Big),
\quad x\in \bQ_p,
\end{equation}
may be considered as a {\it refinement equation}. A solution $\phi$ to this
equation ({\it a refinable function}) is the characteristic function
$\Omega\big(|x|_p\big)$ of the unit disc.
The equation (\ref{62.0-3}) reflects a {\it natural} ``self-similarity'' of
the space $\bQ_p$: the unit disc $B_{0}(0)=\{x: |x|_p \le 1\}$ is represented
as a sum of $p$ mutually {\it disjoint} discs
$B_{0}(0)=B_{-1}(0)\cup\Big(\cup_{r=1}^{p-1}B_{-1}(r)\Big)$,
where $B_{-1}(r)=\bigl\{x: |x-r|_p \le p^{-1}\bigr\}$
(see~\cite[I.3,Examples 1,2.]{Vl-V-Z}).
The equation  (\ref{62.0-3}) is an analog of the {\em refinement equation}
generating the Haar MRA in the real analysis.
Using this idea, the notion of $p$-adic MRA was introduced and a general
scheme for its construction was described in~\cite{S-Sk-1}.
In~\cite{S-Sk-1}, this scheme was realized for construction $2$-adic Haar MRA
with using (\ref{62.0-3}) as the generating refinement equation.
In contrast to the real setting, the {\it refinable function} $\phi$ generating
the Haar MRA is {\em periodic}, which {\em never holds} for real
refinable functions. Doe to this fact, there exist {\em infinity many different}
orthonormal wavelet bases in the same Haar MRA. One of them coincides with Kozyrev's
wavelet basis.
The present paper is devoted to study of $p$-adic {\it refinement equations}
generating MRAs and a method for the construction of MRA-based wavelets.

Here and in what follows, we shall use the notations and the results from~\cite{Vl-V-Z}.
Let $\bN$, $\bZ$, $\bC$ be the sets of positive integers, integers,
complex numbers, respectively.
The field $\bQ_p$ of $p$-adic numbers is defined as the completion
of the field of rational numbers $\bQ$ with respect to the
non-Archimedean $p$-adic norm $|\cdot|_p$. This $p$-adic norm
is defined as follows: $|0|_p=0$; if $x\ne 0$, $x=p^{\gamma}\frac{m}{n}$,
where $\gamma=\gamma(x)\in \bZ$
and the integers $m$, $n$ are not divisible by $p$, then
$|x|_p=p^{-\gamma}$.
The norm $|\cdot|_p$ satisfies the strong triangle inequality
$|x+y|_p\le \max(|x|_p,|y|_p)$.
The canonical form of any $p$-adic number $x\ne 0$ is
\begin{equation}
\label{2}
x=p^{\gamma}(x_0 + x_1p + x_2p^2 + \cdots),
\end{equation}
where $\gamma=\gamma(x)\in \bZ$, \ $x_j=0,1,\dots,p-1$, $x_0\ne 0$,
$j=0,1,\dots$. We shall write the $p$-adic numbers
$k=k_{0}+k_{1}p+\cdots+k_{s-1}p^{s-1}$, $k_0,\dots,k_{s-1}=0,1,2,\dots,p-1$,
in the usual (for the real analysis) form: $k=0,1,\dots,p^s-1$.
Denote by $B_{\gamma}(a)=\{x\in \bQ_p: |x-a|_p \le p^{\gamma}\}$
the disc of radius $p^{\gamma}$ with the center at a point $a\in \bQ_p$,
$\gamma \in \bZ$.
Any two balls in $\bQ_p$ either are disjoint or one
contains the other.

There exists the Haar measure $dx$ on $\bQ_p$,which is  positive,
invariant under the shifts, i.e., $d(x+a)=dx$, and normalized by
$\int_{|\xi|_p\le 1}\,dx=1$.
A complex-valued function $f$ defined on $\bQ_p$ is called
{\it locally-constant} if for any $x\in \bQ_p$ there exists
an integer $l(x)\in \bZ$ such that
$f(x+y)=f(x)$, $y\in B_{l(x)}(0)$.
Denote by ${\cD}(\bQ_p)$ the linear spaces of locally-constant compactly
supported functions (so-called test functions)~\cite[VI.1.,2.]{Vl-V-Z}.
The Fourier transform of $\varphi\in {\cD}(\bQ_p)$ is defined by
${\widehat\phi}(\xi)=F[\varphi](\xi)=\int_{\bQ_p}\chi_p(\xi\cdot x)\varphi(x)\,dx$,
$\xi \in \bQ_p$,
where $\chi_p(\xi\cdot x)=e^{2\pi i\{\xi x\}_p}$ is the additive character
the field $\bQ_p$, $\{\cdot\}_p$ is a fractional part of a number $x\in \bQ_p$.
 The Fourier transform is extended to $L^2(\bQ_p)$ in a
standard way. If $f\in L^2(\bQ_p)$, $0\ne a\in \bQ_p$, \ $b\in \bQ_p$,
then~\cite[VII,(3.3)]{Vl-V-Z}:
\begin{equation}
\label{14}
F[f(ax+b)](\xi)
=|a|_p^{-1}\chi_p\Big(-\frac{b}{a}\xi\Big)F[f(x)]\Big(\frac{\xi}{a}\Big).
\end{equation}
According to~\cite[IV,(3.1)]{Vl-V-Z},
\begin{equation}
\label{14.1}
F[\Omega(p^{-k}|\cdot|_p)](x)=p^{k}\Omega(p^k|x|_p), \quad k\in \bZ,
\quad x \in \bQ_p,
\end{equation}
where $\Omega(t)=1$ for $t\in [0,\,1]$; $\Omega(t)=0$ for $t\not\in [0,\,1]$.

\section{Multiresolution analysis}
\label{s2}

Let
$I_p=\{a=p^{-\gamma}\big(a_{0}+a_{1}p+\cdots+a_{\gamma-1}p^{\gamma-1}\big):
\gamma\in \bN; a_j=0,1,\dots,p-1; j=0,1,\dots,\gamma-1\}$.
It is well known that
$\bQ_p=B_{0}(0)\cup\cup_{\gamma=1}^{\infty}S_{\gamma}$, where
$S_{\gamma}=\{x\in \bQ_p: |x|_p = p^{\gamma}\}$. Due to
(\ref{2}), $x\in S_{\gamma}$, $\gamma\ge 1$, if and only if
$x=x_{-\gamma}p^{-\gamma}+x_{-\gamma+1}p^{-\gamma+1}+\cdots+x_{-1}p^{-1}+\xi$,
where $x_{-\gamma}\ne 0$, $\xi \in B_{0}(0)$. Since
$x_{-\gamma}p^{-\gamma}+x_{-\gamma+1}p^{-\gamma+1}
+\cdots+x_{-1}p^{-1}\in I_p$, we have a ``natural'' decomposition of
$\bQ_p$ to a union of mutually  disjoint discs:
$\bQ_p=\cup_{a\in I_p}B_{0}(a)$.
So, $I_p$ is a {\em ``natural'' set of shifts} for $\bQ_p$.

\begin{definition}
\label{de1} \rm
(~\cite{S-Sk-1})
A collection of closed spaces
$V_j\subset L^2(\bQ_p)$, $j\in\bZ$, is called a
{\it multiresolution analysis {\rm(}MRA{\rm)} in $ L^2(\bQ_p)$} if the
following axioms hold

(a) $V_j\subset V_{j+1}$ for all $j\in\bZ$;

(b) $\bigcup_{j\in\bZ}V_j$ is dense in $ L^2(\bQ_p)$;

(c) $\bigcap_{j\in\bZ}V_j=\{0\}$;

(d) $f(\cdot)\in V_j \Longleftrightarrow f(p^{-1}\cdot)\in V_{j+1}$
for all $j\in\bZ$;

(e) there exists a function $\phi \in V_0$
such that the system $\{\phi(\cdot-a), a\in I_p\}$ is an orthonormal
basis for $V_0$.
\end{definition}

The function $\phi$ from axiom (e) is called {\em scaling}.
It follows  immediately
from axioms (d) and (e) that the functions $p^{j/2}\phi(p^{-j}\cdot-a)$,
$a\in I_p$, form an orthonormal basis for $V_j$, $j\in\bZ$.
According to the standard scheme (see, e.g.,~\cite[\S 1.3]{NPS})
for construction of MRA-based wavelets, for each $j$, we define
a space $W_j$ ({\em wavelet space}) as the orthogonal complement
of $V_j$ in $V_{j+1}$, i.e., $V_{j+1}=V_j\oplus W_j$, $j\in \bZ$,
where $W_j\perp V_j$, $j\in \bZ$. It is not difficult to see that
\begin{equation}
\label{61.0}
f\in W_j \Longleftrightarrow f(p^{-1}\cdot)\in W_{j+1},
\quad\text{for all}\quad j\in \bZ
\end{equation}
and $W_j\perp W_k$, $j\ne k$.
Taking into account axioms (b) and (c), we obtain
\begin{equation}
\label{61.1}
{\bigoplus\limits_{j\in\bZ}W_j}= L^2(\bQ_p)
\quad \text{(orthogonal direct sum)}.
\end{equation}

If now we find  a finite number of functions $\psi^{(\nu)} \in W_0, \nu\in A,$
such that the system $\{\psi^{(\nu)}(x-a), a\in~I_p, \nu\in A\}$ is an orthonormal
basis for $W_0$, then, due to~(\ref{61.0}) and (\ref{61.1}),
the system $\{p^{j/2}\psi^{(\nu)}(p^{-j}\cdot-a), a\in I_p, j\in\bZ,\nu\in A\}$,
is an orthonormal basis for $ L^2(\bQ_p)$.
Such  functions $\psi^{(\nu)}$ are called  {\em wavelet functions} and
the basis is a {\em wavelet basis}.

Let $\phi$ be a {\em scaling function} for a MRA. As was mentioned
above, the system $\{p^{1/2}\phi(p^{-1}\cdot-a), a\in I_p\}$ is a
basis for $V_1$. It follows from axiom (a) that
\begin{equation}
\label{62.0-2*} \phi=\sum_{a\in I_p}\beta_a\phi(p^{-1}\cdot-a),
\quad \beta_a\in \bC.
\end{equation}
We see that the function
$\phi$ is a solution of a special kind of functional equation.
Such equations are called {\em refinement equations}, and their
solutions are called  {\em refinable functions} \footnote{Usually
the terms ``scaling function'' and ``refinable function'' are
synonyms in the literature, and they are used for both the
senses: as a solution to a refinement equation and as a function
generating MRA. We separate the meanings of these terms.}.

A natural way for the construction of a MRA (see, e.g.,~\cite[\S 1.2]{NPS})
is the following. We start with an appropriate function $\phi$
whose $I_p$-shifts form an orthonormal system and set
\begin{equation}
\label{62.0-11}
V_j=\overline{{\rm span}\big\{\phi\big(p^{-j}\cdot-a\big):a\in I_p\big\}},
\quad j\in \bZ.
\end{equation}
It is clear that axioms (d) and (e) of Definition~\ref{de1} are fulfilled.
Of course, not any such a function $\phi$ provides axiom $(a)$.
In the {\em real setting}, the relation $V_0\subset V_{1}$ holds
if and only if the scaling function satisfies a refinement equation.
Situation is different in $p$-adics. Generally speaking, a refinement
equation (\ref{62.0-2*}) {\em does not imply the including property}
$V_0\subset V_{1}$ because the set of the shifts $I_p$ {\em does not
form a group}.
Indeed, we need all the functions $\phi(\cdot-b)$,
$b\in I_p$, to belong to the space $V_1$, i.e., the identities
$\phi(x-b)=\sum_{a\in I_p}\alpha_{a,b}\phi(p^{-1}x-a)$ should be
fulfilled for all $b\in I_p$. Since $p^{-1}b+a$ is not in $I_p$ in general,
we {\em can not state} that
$\phi(x-b)=\sum_{a\in I_p}\alpha_{a,b}\phi(p^{-1}x-p^{-1}b-a)\in V_1$
for all $b\in I_p$.
Nevertheless, some refinement equations
imply including property, which may happen because of different causes.

The {\it refinement equation} (\ref{62.0-3}) is a particular case of
(\ref{62.0-2*}).

\section{Construction of refinable functions}
\label{s3}

Now we are going to study $p$-adic refinement equations  and
 their solutions. We restrict ourselves by the refinement
equations (\ref{62.0-2*}) with a finite number of the terms in the right-hand side:
\begin{equation}
\label{62.0-5}
\phi(x)=\sum_{k=0}^{p^s-1}\beta_{k}\phi\Big(\frac{1}{p}x-\frac{k}{p^s}\Big).
\end{equation}
If $\phi\in  L^2(\bQ_p)$, taking the Fourier transform and using
(\ref{14}), one can rewrite (\ref{62.0-5}) as
\begin{equation}
\label{62.0-6}
{\widehat\phi}(\xi)=m_0\Big(\frac{\xi}{p^{s-1}}\Big){\widehat\phi}(p\xi),
\end{equation}
where
\begin{equation}
\label{62.0-7-1}
m_0(\xi)=\frac{1}{p}\sum_{k=0}^{p^s-1}\beta_{k}\chi_p(k\xi)
\end{equation}
is a trigonometric polynomial. It is clear that $m_0(0)=1$ whenever $\widehat\phi(0)\ne0$.

\begin{theorem}
\label{th1-1*}
If $\phi$ is a refinable function such that ${\rm supp\,{\widehat\phi}}\subset B_0(0)$
and the system $\{\phi(x-a):a\in I_p\}$ is orthonormal, then axiom $(a)$ from
Definition~{\rm\ref{de1}} holds for the spaces {\rm(\ref{62.0-11})}.
\end{theorem}

\begin{proof}
Since $\chi_p(-\xi)=1$ for $\xi\in B_0(0)$, we have
$\chi_p(-\xi){\widehat\phi}(\xi)={\widehat\phi}(\xi)$. Applying
the Fourier transform, we obtain $\phi(x+1)=\phi(x)$. Thus $\phi$
is a {\em $1$-periodic} function.
Since for all $a\in I_p$ and all $k=0,1,\dots,p^s-1$,
either $\frac{a}{p}+\frac{k}{p^s}\in I_p$, or
$\frac{a}{p}+\frac{k}{p^s}-1\in I_p$. Due to the {\em $1$-periodicity} of
$\phi$, it follows from (\ref{62.0-5}) that $\phi(x-b)\in V_1$ for all
$b\in I_p$. This implies $V_0\subset V_{1}$, similarly
$V_j\subset V_{j+1}$ for any $j\in \bZ$.
\end{proof}

\begin{theorem}
Let $\phi\in L^2(\bQ_p)$ be a refinable function, the system $\{\phi(x-a):a\in I_p\}$
be orthonormal and ${\rm supp\,{\widehat\phi}}\subset B_0(0)$.
Axiom $(b)$ of Definition~{\rm\ref{de1}} holds for the spaces
{\rm(\ref{62.0-11})} {\rm(}i.e., $\overline{\cup_{j\in\bZ}V_j}= L^2(\bQ_p)${\rm)}
if and only if
\be
\bigcup\limits_{j\in\bZ}{\rm supp\,}\widehat\phi(p^{j}\cdot)=\bQ_p.
\label{dnn14}
\ee
\label{th1-3*}
\end{theorem}

\begin{proof}
First of all we note that, due to axioms $(d)$ and $(e)$,
each space $V_j$ is invariant with respect to the shifts
$t=p^{j}a$, $a\in I_p$. Show that the space
$\overline {\cup_{j\in {\bZ}} V_j}$ is invariant with respect to any
shift  $t\in\bQ_p$. Every $t\in\bQ_p$ may be approximated by
a vector  $p^{j}a$, $a\in I_p$, with arbitrary large $j\in\bZ$.
If $f\in \cup_{j\in\bZ} V_j$, by axiom  $(a)$ (which holds due to
Theorem~\ref{th1-1*}), then
$f\in V_j$ for all $j\ge j_1$. It follows from the continuity of
the function $\|f(\cdot+t)\|_2$ that
$f(\cdot + t) \in \overline{\cup_{j\in\bZ} V_j}$.
Now let $t\in\bQ_p$. If $g\in\overline{\cup_{j\in\bZ} V_j}$, then
approximating $g$ by the functions  $f\in \cup _{j\in\bZ} V_j$,
again using the continuity of the shift operator and the invariance
of $L_2$ norm with respect to the shifts, we derive
$g(\cdot + t) \in \overline{\cup_{j\in\bZ} V_j}$.
For $X\subset L^2(\bQ_p)$, set $\widehat X=\{\ \widehat f: f\in X\}$.
By the Wiener theorem for $L_2$ (see, e.j., \cite{NPS}, all the arguments
of the proof given there may be repeated word for word with replacing
 $\bR$ by $\bQ_p$), a closed subspace
$X$ of the space $L^2(\bQ_p)$ is invariant with respect to the shifts
if and only if $\widehat X=L_2(\Omega)$ for some set
$\Omega\subset\bQ_p$. Let $X=\overline{\cup_{j\in\bZ}V_j}$, then
$\widehat X=L_2(\Omega)$. Thus $X=L^2(\bQ_p)$ if and only if $\Omega=\bQ_p$. Set
$\phi_j=\phi(p^{-j}\cdot),\ \ \Omega_0=\cup_{j\in\bZ}{\rm supp}\, \widehat\phi_j$
and prove that $\Omega=\Omega_0$. Since $\phi_j\in V_j,$ $j\in\bZ$,
we have ${\rm supp}\, \widehat\phi_j\subset\Omega$, and hence
$\Omega_0 \subset \Omega$.
Now assume that $\Omega\backslash\Omega_0$ contains a set of
positive measure $\Omega_1$. If $f\in V_j$, taking the Fourier
transform from the expansion $f=\sum_{a\in I_p} h_a \phi(p^{-j}\cdot -a)$
we see that $\widehat f=0$ almost everywhere  on $\Omega_1.$
Hence the same is true for any $f\in \cup _{j\in\bZ} V_j$.
Passing to the limit we deduce that that the Fourier transform
of any $f\in X$ is equal to zero almost everywhere on  $\Omega_1$, i.e.,
$L_2(\Omega)=L_2(\Omega_0)$.
It remains to note that
${\rm supp\,} \widehat\phi_j={\rm supp\,}\widehat\phi({p}^{j}\cdot) $
\end{proof}

\begin{theorem}
\label{th1-4*}
If $\phi \in  L^2(\bQ_p)$ and the system $\{\phi(x-a):a\in I_p\}$ is orthonormal,
then axiom $(c)$ of Definition~{\rm\ref{de1}} holds, i.e., $\cap_{j\in\bZ}V_j=\{0\}$.
\end{theorem}

\begin{proof}
First, using the standard scheme (see, e.g.,~\cite[Lemma~1.2.8.]{NPS}), we prove
that for any $f\in  L^2(\bQ_p)$
\begin{equation}
\label{62.0-11-1}
\lim_{j\to -\infty}\sum_{a\in I_p}
\big|\big(f,p^{j/2}\phi(p^{-j}\cdot-a)\big)\big|^2=0,
\end{equation}
where $(\cdot,\,\cdot)$ is the scalar product in $ L^2(\bQ_p)$.
Since the space $\cD(\bQ_p)$ is dense in $ L^{2}(\bQ_p)$~\cite[VI.2]{Vl-V-Z},
it suffices to prove (\ref{62.0-11-1}) for any $\varphi\in \cD(\bQ_p)$.
If $\varphi\in \cD(\bQ_p)$, then there exists $N$ such that
$\varphi(x)=0$ for all $|\xi|_p>p^{N}$. Since $|\varphi(x)|\le M$
for all $|\xi|_p\le p^{N}$, we have
$$
\sum_{a\in I_p}\big|\big(\varphi(\cdot),p^{j/2}\phi(p^{-j}\cdot-a)\big)\big|^2
\le p^{j}\sum_{a\in I_p}\bigg(\int_{|x|_p\le p^{N}}|\varphi(x)||\phi(p^{-j}x-a)|\,dx\bigg)^2
$$
$$
\qquad\qquad\qquad\qquad\qquad\qquad
\le p^{j+N}M^2\sum_{a\in I_p}\int_{|x|_p\le p^{N}}|\phi(p^{-j}x-a)|^2\,dx.
$$
By the change of variables $\eta=p^{-j}x-a$, we obtain
$$
\sum_{a\in I_p}\big|\big(\varphi,p^{j/2}\phi(p^{-j}\cdot-a)\big)\big|^2
\le M^2p^{N}\int\limits_{A_{Nj}}|\phi(\eta)|^2\,d\eta
=M^2p^{N}\int\limits_{\bQ_p}\theta_{Nj}(\eta)|\phi(\eta)|^2\,d\eta,
$$
where $\theta_{Nj}$ is the characteristic function of the set
$A_{Nj}=\cup_{a\in I_p}\{\eta:|\eta+a|_p\le 2^{N+j}\}$.
Since $\lim_{j\to -\infty}\theta_{Nj}(\eta)=0$ for any $\eta\ne -a$, using
the Lebesgue dominated convergence theorem~\cite[IV.4]{Vl-V-Z}, we obtain
$\lim_{j\to-\infty}\int_{\bQ_p}\theta_{Nj}(\eta)|\phi(\eta)|^2\,d\eta=0$.

If now we assume that $f\in \cap_{j\in\bZ}V_j$, then  $f\in V_j$ for
all $j\in\bZ$ and, due to (\ref{62.0-11-1}),
$$
\|f\|=\bigg(\sum_{a\in I_p}\big|\big(f,p^{j/2}\phi(p^{-j}\cdot-a)\big)\big|^2\bigg)^{1/2}
\to 0, \quad j\to -\infty,
$$
i.e., $\|f\|=0$ which was to be proved.
\end{proof}

\begin{theorem}
\label{th1-2*}
Let $\phi$ be a refinable function such that ${\rm supp\,{\widehat\phi}}\subset B_0(0)$.
If $|{\widehat\phi}(\xi)|=1$ for all $\xi\in B_0(0)$ then the system
$\{\phi(x-a):a\in I_p\}$ is orthonormal.
\end{theorem}

\begin{proof}
Taking into account formula (\ref{14.1}),  using  the inclusion
${\rm supp\,{\widehat\phi}}\subset B_0(0)$ and the Plancherel
formula, we have for any $a\in I_p$
$$
\big(\phi(\cdot),\phi(\cdot-a)\big)=\int_{\bQ_p}\phi(x){\overline\phi}(x-a)\,dx
=\int_{B_0(0)}|{\widehat\phi}(\xi)|^2\chi_p(a\xi)\,d\xi
$$
$$
\qquad
=\int_{B_0(0)}\chi_p(a\xi)\,d\xi=\int_{\bQ_p}\Omega(|\xi|_p)\chi_p(a\xi)\,d\xi
=\Omega(|a|_p)=\delta_{a 0}.
$$
\end{proof}

So, to construct a MRA we can take a  function $\phi$ for which the hypotheses
of Theorems~\ref{th1-2*} and~(\ref{dnn14}) are fulfilled. Next
we are going to describe all such functions.

\begin{proposition}
\label{pr1-1}
If $\phi\in  L^2(\bQ_p)$ is a solution of refinable
equation~(\ref{62.0-6}), ${\widehat\phi}(\xi)$ is continuous at
the point $0$ and ${\widehat\phi}(0)\ne 0$, then
\begin{equation}
\label{62.0-8}
{\widehat\phi}(\xi)={\widehat\phi}(0)\prod_{j=1}^{\infty}m_0\Big(\frac{\xi}{p^{s-j}}\Big).
\end{equation}
\end{proposition}

\begin{proof}
Iterating {\rm(\ref{62.0-6})} $N$ times, $N\ge 1$, we have
\begin{equation}
\label{62.0-9}
{\widehat\phi}(\xi)=\prod_{j=1}^{N}m_0\Big(\frac{\xi}{p^{s-j}}\Big)
{\widehat\phi}(p^{N}\xi).
\end{equation}
Taking into account that ${\widehat\phi}(\xi)$ is continuous at
the point $0$ and the fact that $|p^{N}\xi|_p=p^{-N}|\xi|_p\to 0$
as $N\to +\infty$ for any $\xi\in\bQ_p$, we obtain {\rm(\ref{62.0-8})}.
\end{proof}

\begin{proposition}
\label{pr1-2}
If ${\widehat\phi}$ is defined by {\rm(\ref{62.0-8})}, where $m_0$ is
a trigonometric polynomial {\rm(\ref{62.0-7-1})}, $m_0(0)=1$, then
{\rm(\ref{62.0-6})} holds. Furthermore, if $\xi\in \bQ_p$ such that $|\xi|_p=p^{-n}$,
then  ${\widehat\phi}(\xi)={\widehat\phi}(0)$ for $n\ge s-1$,  and
\begin{equation}
\label{62.0-10}
{\widehat\phi}(\xi)={\widehat\phi}(0)
\prod_{j=1}^{s-n-1}m_0\Big(\frac{\xi}{p^{s-j}}\Big)
={\widehat\phi}(0)\prod_{j=1}^{s-n-1}m_0\Big(\frac{\tilde{\xi}}{p^{s-j}}\Big)
={\widehat\phi}(\tilde{\xi})
\end{equation}
where $\tilde{\xi}=\xi_{n}p^{n}+\xi_{n+1}p^{n+1}+\cdots+\xi_{s-2}p^{s-2}$,
$\xi_{n}\ne 0$, for $n\le s-2$.
\end{proposition}

\begin{proof}
Relation {\rm(\ref{62.0-8})} implies ${\widehat\phi}(p\xi)
={\widehat\phi}(0)\prod_{j=1}^{\infty}m_0\big(\frac{\xi}{p^{s-j-1}}\big)$
and, consequently, {\rm(\ref{62.0-6})}.
Let  $n\le s-2$,  $|\xi|_p=p^{-n}$, i.e., $\xi=\xi_{n}p^{n}+\xi_1p^{n+1}+\cdots$,
$\xi_{n}\ne 0$. Since $\chi_p\big(\frac{k\xi'}{p^{s-j}}\big)=1$,
whenever $\xi'\in B_{-s+1}(0)$, $j\in\bN$, $k=0,1,\dots,p^s-1$,
and $\chi_p\big(\frac{k\xi}{p^{s-j}}\big)=1$,  whenever
$j\ge s-n$, we have (\ref{62.0-10}).
It is clear that ${\widehat\phi}(\xi)={\widehat\phi}(0)$ for $n\ge s-1$.
\end{proof}

\begin{corollary}
\label{cor5-1}
The function ${\widehat\phi}$ from Proposition~{\rm\ref{pr1-2}} is
locally-constant. Moreover, if $M\ge -s+2$, \ ${\rm supp}{\widehat\phi}\subset B_{M}(0)$, then
for any $k=0,1,\dots,p^{M+s-1}-1$ and for all $x\in B_{s-1}\big(\frac{k}{p^M}\big)$
we have ${\widehat\phi}(x)={\widehat\phi}\big(\frac{k}{p^M}\big)$.
\end{corollary}

\begin{proposition}
\label{pr1-3}
Let ${\widehat\phi}$ be defined by {\rm(\ref{62.0-8})}, where $m_0$ is
a trigonometric polynomial~{\rm(\ref{62.0-7-1})}.
If $m_0(0)=1$, $m_0\big(\frac{k}{p^{s}}\big)=0$ for all $k=1,\dots,p^s-1$
which are not divisible by $p$,
then ${\rm supp\,{\widehat\phi}}\subset B_0(0)$, ${\widehat\phi}\in  L^2(\bQ_p)$.
If, furthermore,  $\big|m_0\big(\frac{k}{p^{s}}\big)\big|=1$ for all $k=1,\dots,p^s-1$
which are divisible by $p$, then $|{\widehat\phi}(x)|=|{\widehat\phi}(0)|$
for any $x\in B_0(0)$.
\end{proposition}

\begin{proof}
By Proposition~\ref{pr1-2}, ${\widehat\phi}$ satisfyes (\ref{62.0-6}) and (\ref{62.0-10}).
Let us check that ${\widehat\phi}(\xi)=0$ for all $\xi$ such that $|\xi|_p=p^{M}$,
$M\ge 1$. It follows from (\ref{62.0-6}) that it suffices to consider only $M=1$.
Let $|\xi|_p=p$, i.e., $\xi=\frac{1}{p}\xi_{-1}+\xi_0 +\xi_1p+\dots+\xi_{s-2}p^{s-2}+\xi'$,
where $\xi_{-1}\ne 0$, $\xi'\in B_{-s+1}(0)$.
In view of (\ref{62.0-10}), we have
${\widehat\phi}(\xi)={\widehat\phi}(\tilde{\xi})={\widehat\phi}\big(\frac{k}{p}\big)$,
where $k=\xi_{-1}+\xi_0p+\xi_1p^2+\dots+\xi_{s-2}p^{s-1}$, $\xi_{-1}\ne 0$.
Note that the first factor of the product in (\ref{62.0-10}) is $m_0\big(\frac{k}{p^{s}}\big)=0$.
Thus ${\widehat\phi}(\xi)=0$ for all $\xi$ such that $|\xi|_p=p$.
The rest statements follow from Propositions~\ref{pr1-1},~\ref{pr1-2}.
\end{proof}

Due to  Theorems~\ref{th1-1*}-\ref{th1-2*}, the
refinable functions with masks satisfying the hypotheses of
Proposition~\ref{pr1-3}  generate MRAs.
Next, we will see that all properties of a mask $m_0$ described in
Propositions~\ref{pr1-3} are necessary for the corresponding
refinable function $\phi$ to be such that ${\rm supp\,{\widehat\phi}}\subset B_0(0)$
and the system $\{\phi(x-a):a\in I_p\}$ is orthonormal.

\begin{theorem}
\label{th1-5*}
Let ${\widehat\phi}$ be defined by {\rm(\ref{62.0-8})}, where $m_0$ is
a trigonometric polynomial~{\rm(\ref{62.0-7-1})}.
If  ${\rm supp\,{\widehat\phi}}\subset B_0(0)$ and the system $\{\phi(x-a):a\in I_p\}$
is orthonormal, then $\big|m_0\big(\frac{k}{p^{s}}\big)\big|=0$ whenever $k$ is not divisible
by $p$, and $\big|m_0\big(\frac{k}{p^{s}}\big)\big|=1$ whenever $k$ is divisible by $p$,
\ $k=1,2,\dots,p^s-1$.
\end{theorem}

\begin{proof}
Let $a\in I_p$. Due to the orthonormality of $\{\phi(x-a):a\in I_p\}$,
using the Plancherel formula and Corollary~\ref{cor5-1}, we have
$$
\delta_{a 0}=
\big(\phi(\cdot),\phi(\cdot-a)\big)=\int_{\bQ_p}\phi(x){\overline\phi}(x-a)\,dx
=\int_{B_0(0)}|{\widehat\phi}(\xi)|^2\chi_p(a\xi)\,d\xi
$$
$$
=\sum_{k=0}^{p^{s-1}-1}\int_{|\xi-k|_p\le p^{-s+1}}|{\widehat\phi}(\xi)|^2\chi_p(a\xi)\,d\xi
=\sum_{k=0}^{p^{s-1}-1}|{\widehat\phi}(k)|^2\int_{|\xi-k|_p\le p^{-s+1}}\chi_p(a\xi)\,d\xi
$$
$$
=\sum_{k=0}^{p^{s-1}-1}|{\widehat\phi}(k)|^2\chi_p(ak)
\int\limits_{|\xi|_p\le p^{-s+1}}\chi_p(a\xi)
\,d\xi
=\frac{1}{p^{s-1}}\Omega(|p^{s-1}a|_p)
\sum_{k=0}^{p^{s-1}-1}|{\widehat\phi}(k)|^2\chi_p(ak).
$$

Since $\Omega(|p^{s-1}a|_p)\ne 0$ for all $a\in B_{-s+1}(0)$, this yields
$$
\frac{1}{p^{s-1}}\sum_{k=0}^{p^{s-1}-1}|{\widehat\phi}(k)|^2\chi_p(ak)=\delta_{a 0},
\quad a=0,\frac{1}{p^{s-1}},\dots,\frac{p^{s-1}-1}{p^{s-1}}.
$$

Consider these equalities as a linear system with respect to the unknowns
$z_k=|{\widehat\phi}(k)|^2$. It is well-known that the system has a unique
solution $z_k=1$, $k=0,1,\dots,p^{s-1}-1$, i.e., $|{\widehat\phi}(k)|=1$
for all $k=0,1,\dots,p^{s-1}-1$. In particular, it follows that $|{\widehat\phi}(0)|^2=1$,
which reduces (\ref{62.0-8}) to
\begin{equation}
\label{62.0-15}
{\widehat\phi}(\xi)=m_0\Big(\frac{\xi}{p^{s-1}}\Big)m_0\Big(\frac{\xi}{p^{s-2}}\Big)
\cdots m_0\big(\xi\big).
\end{equation}

Let us check that $\big|m_0\big(\frac{k}{p^{s}}\big)\big|=1$ for all $k$  divisible by $p$,
$k=1,2,\dots,p^{s-1}-1$. This is equivalent to $\big|m_0\big(\frac{k}{p^{N}}\big)\big|=1$
whenever $N=1,2,\dots,p^{s-1}$, $k=1,2,\dots,p^{N}-1$, $k$ is not divisible by $p$.
We will prove this statement by induction on $N$. For the inductive base with $N=1$,
note that $1=|{\widehat\phi}(p^{s-2})|=|m_0(\frac{1}{p})|$. For the inductive step,
assume that $\big|m_0\big(\frac{k}{p^{n}}\big)\big|=1$ for all $n=1,2,\dots,N$,
$N\le s-2$, $k=1,2,\dots,p^{n}-1$, $k$ is not divisible by $p$.
Using (\ref{62.0-15}) and $1$-periodicity of $m_0$, we have
$1=\Big|{\widehat\phi}\Big(lp^{s-N-2}\Big)\Big|
=\Big|m_0\Big(\frac{l}{p^{N+1}}\Big)m_0\Big(\frac{l}{p^{N}}\Big)
\cdots m_0\Big(\frac{l}{p^{N-s+2}}\Big)\Big|
=\Big|m_0\Big(\frac{l}{p^{N+1}}\big)\Big|$,
for all $l=1,2,\dots,p^{N+1}-1$, $l$ is not divisible by $p$.
Now  assume that $m_0\big(\frac{k}{p^{s}}\big)\ne0$ for some $k=1,2,\dots,p^{s}-1$
 not divisible by $p$. Since ${\widehat\phi}(\frac{k}{p})=0$, due to (\ref{62.0-15}),
there exists $n=1,2,\dots,s$ such that $m_0\big(\frac{k}{p^{s-n}}\big)=0$ which
contradicts to $\big|m_0\big(\frac{k}{p^{s-n}}\big)\big|=1$.
\end{proof}

We have investigated refinable functions whose Fourier transform
is supported in the unit disk $B_0(0)$. Such functions provide
axiom $(a)$ of Definition\ref{de1} because of a trivial argument given
in Theorem~\ref{th1-1*}.
If ${\rm supp\,}\widehat\phi\not\subset B_0(0)$, generally speaking,
the relation $\phi(\cdot-a)\in V_1$ does not follow from the
refinability of $\phi$ for all $a\in I_p$ because $I_p$ is not a group.
Nevertheless, we observed that some such refinable functions also
provide axiom $(a)$. Let $p=2$, $s=3$,  $\phi$ be defined by (\ref{62.0-8}),
where $m_0$ is given by (\ref{62.0-7-1}), $m_0(1/4)=m_0(3/8)=m_0(7/16)=m_0(15/16)=0$.
It is not difficult to see that ${\rm supp\,}\widehat\phi\subset B_1(0)$,
${\rm supp\,}\widehat\phi\not\subset B_0(0)$.
Evidently, axiom  $(a)$ will be fulfilled whenever
$\phi\Big(x-\frac{k}{4}\Big)=\sum_{r=0}^7\gamma_{kr}\phi\Big(\frac{1}{2}x-\frac{r}{8}\Big)$,
$k=1,2,3$, \ $x\in \bQ_2$, which is equivalent to
${\widehat\phi}(\xi)\chi_2\Big(\frac{k\xi}{4}\Big)
=m_k\Big(\frac{\xi}{4}\Big){\widehat\phi}(2\xi)$,
$k=1,2,3$, \ $\xi\in \bQ_2$,
where $m_k(\xi)=\frac{1}{2}\sum\limits_{r=0}^7\gamma_{k,r}\chi_2(r\xi)$.
Combining this with (\ref{62.0-6}) we have
$\widehat\phi(8\xi)(m_0(\xi)\chi_2(k\xi))-m_k(\xi))=0$, $k=1,2,3$,
\ $\xi\in \bQ_2$.
These equalities will be fulfilled for any $\xi\in \bQ_2$ whenever
they are fulfilled for $\xi=l/16$, $l=0,1,\dots, 15$.
Desirable polynomials $m_k$, $k=1,2,3$, exist because we have
$\widehat\phi\Big(\frac{1}{2}\Big)=\widehat\phi\Big(\frac{3}{2}\Big)=
\widehat\phi\Big(\frac{5}{2}\Big)=\widehat\phi\Big(\frac{9}{2}\Big)=
\widehat\phi\Big(\frac{11}{2}\Big)=\widehat\phi\Big(\frac{13}{2}\Big)=
\widehat\phi(1)=\widehat\phi(5)=0$.
So, we succeeded with providing axiom $(a)$ of Definition\ref{de1}, but,
unfortunately, such a $\phi$ is not a scaling function generating MRA
because axiom $(e)$ is not valid.
Moreover, it is possible to show that for any refinable function whose
Fourier transform is in $B_1(0)$ but not in $B_0(0)$ the shift system
$\{\phi(x-a):a\in I_p\}$ is not orthogonal.
We suggest the following conjecture: {\em it does not exist compactly
supported refinable functions with mutually orthogonal shifts
$\{\phi(x-a):a\in I_p\}$ whose Fourier transform is not supported in $B_0(0)$}.

\section{Wavelet bases}
\label{s4}

Now we discuss how to find wavelet functions if we have already a $p$-adic
MRA generating by scaling function. Let the refinement equation for $\phi$
be (\ref{62.0-5}). We look for wavelet functions $\psi^{(\nu)}$,
$\nu=1,\dots,p-1$, in the form
$\psi^{(\nu)}(x)=\sum_{k=0}^{p^s-1}\gamma_{\nu k}\phi\Big(\frac{1}{p}x-\frac{k}{p^s}\Big)$,
where the coefficients $\gamma_{\nu k}$ are chosen such that
\begin{equation}
\label{62.0-25}
(\psi^{(\nu)}, \phi(\cdot-a))=0,
\quad
(\psi^{(\nu)}, \psi^{(\mu)}(\cdot-a))=\delta_{\nu \mu}\delta_{0 a},
\quad
\nu,\mu=1,\dots,p-1,
\end{equation}
for any $a\in I_p$.
It is clear that (\ref{62.0-25}) are fulfilled for all
$a\ne0,\frac{1}{p^{s-1}},\dots,\frac{p^{s-1}-1}{p^{s-1}}$.

Set $B=\frac{1}{\sqrt{p}}(\beta_0,\dots,\beta_{p^{s}-1})^T$,
$G_{\nu}=\frac{1}{\sqrt{p}}(\gamma_{\nu 0},\dots,\gamma_{\nu, p^{s}-1})^T$,
$\nu=1,\dots,p-1$,
$$
S=\left(
\begin{array}{ccccc}
0&0&\ldots&0&1 \\
1&0&\ldots&0&0 \\
0&1&\ldots&0&0 \\
\hdotsfor{5} \\
0&0&\ldots&1&0 \\
\end{array}
\right).
$$
To provide (\ref{62.0-25}) for $a=0,\frac{1}{p^{s-1}},\dots,\frac{p^{s-1}-1}{p^{s-1}}$
we should find vectors $G_{1},\dots,G_{p-1}$ so that the matrix
$$
U=(S^{0}B,\dots,S^{p^{s-1}-1}B,S^{0}G_{1},\dots,S^{p^{s-1}-1}G_{1},\dots,
S^{0}G_{p-1},\dots,S^{p^{s-1}-1}G_{p-1})
$$
is unitary.

\begin{example}
Let $s=1$. According to Proposition~\ref{pr1-3},
we set $m_0(0)=1$, $m_0\big(\frac{k}{p}\big)=0$, $k=1,\dots,p-1$ and find
the mask $m_0(\xi)=\frac{1}{p}\sum_{k=0}^{p-1}\chi_p(k\xi)$ (here $\beta_{k}=1$ in
(\ref{62.0-7-1}) and $B=\frac{1}{\sqrt{p}}(1,\dots,1)^T$). The corresponding
refinement equation~(\ref{62.0-5}) coincides with the ``natural'' refinement
equation (\ref{62.0-3}), and its solution $\Omega\big(|\cdot|_p\big)$ is a refinable
function generating a MRA because of  Theorems~\ref{th1-1*}-\ref{th1-2*}.
To find wavelets we observe that the unitary matrix
$\{\frac{1}{\sqrt p}e^{2\pi ikl}\}_{k,l=0,\dots,p-1}$ may be taken as $U^T$.
Computing the wavelet functions corresponding to this matrix $U$, we derive
the formulas which were found in~\cite{Koz0}:
$\psi^{(\nu)}(x)=\chi_p\big(\frac{\nu}{p}x\big)\Omega\big(|x|_p\big)$,
$\nu=1,\dots,p-1$.
\end{example}

\begin{example}
 Let $s=2$, $p=3$. According to Proposition~\ref{pr1-3},
we set $m_0\big(\frac{k}{9}\big)=0$ if $k$ is not divisible by $3$,
and $m_0(0)=1$, $m_0\big(\frac{1}{3}\big)=m_0\big(\frac{2}{3}\big)=-1$.
In this case  $m_0(z)=3^{-2}\big(-1+2z+2z^2-z^3+2z^4+2z^5-z^6+2z^7+2z^8\big)$,
$z=e^{2\pi i\xi}$ and
$$
{\widehat\phi}(\xi)=\left\{
\begin{array}{rll}
1,  && |\xi|_p\le \frac{1}{3}, \smallskip \\
-1, && |\xi-1|_p\le \frac{1}{3}, \smallskip \\
-1, && |\xi-2|_p\le \frac{1}{3}, \\
0, && |\xi|_p\ge 3. \\
\end{array}
\right.
$$
Thus, the corresponding refinement equation (\ref{62.0-5}) is
$\phi(x)=\sum_{k=0}^{8}\beta_{k}\phi\big(\frac{x}{3}-\frac{k}{9}\big)$,
where $\beta_{0}=-\frac{1}{3}$, $\beta_{1}=\beta_{2}=\frac{2}{3}$,
$\beta_{3}=-\frac{1}{3}$, $\beta_{4}=\beta_{5}=\frac{2}{3}$,
$\beta_{6}=-\frac{1}{3}$, $\beta_{7}=\beta_{8}=\frac{2}{3}$, and
\begin{equation}
\label{70.1}
\phi(x)=\left\{
\begin{array}{rll}
-\frac{1}{3}, && |x|_p\le \frac{1}{3} \smallskip \\
\frac{2}{3},  && |x-\frac{1}{3}|_p\le 1 \smallskip \\
\frac{2}{3},  && |x-\frac{2}{3}|_p\le 1 \smallskip \\
0, && |x|_p\ge 9 \\
\end{array}
\right.
=\frac{1}{3}\Omega(|3x|_p)\big(1-e^{2\pi i\{x\}_3}-e^{4\pi i\{x\}_3}\big).
\end{equation}

So, we have $B=\frac{1}{3\sqrt{3}}(-1,2,2,-1,2,2,-1,2,2)^T$, and the vectors
$B,SB, S^2B$ are  orthonormal. We need to extend these three columns to
a unitary matrix $U=(B,SB, S^2B, G_1,SG_1, S^2G_1, G_2,SG_2, S^2G_2)$.
It is not difficult to see that the vectors
$G_{1}=\frac{1}{\sqrt{3}}(1,0,0,-1,0,0,0,0,0)^T$,
$G_{2}=\frac{1}{\sqrt{3}}(1,0,0,1,0,0,-2,0,0)^T$
are appropriate.
The corresponding wavelet functions are  $\psi^{(1)}=\sqrt{\frac{3}{2}}
\big(\phi\big(\frac{x}{3}\big)-\phi\big(\frac{x}{3}-\frac{1}{3}\big)\big)$,
 $\psi^{(2)}=\frac{1}{\sqrt{2}}
\big(\phi\big(\frac{x}{3}\big)+\phi\big(\frac{x}{3}-\frac{1}{3}\big)
-2\phi\big(\frac{x}{3}-\frac{2}{3}\big)\big)$.
Substituting (\ref{70.1}), we obtain
$$
\psi^{(1)}(x)=\left\{
\begin{array}{rll}
-\sqrt{\frac{3}{2}}, && |x|_p\le \frac{1}{3}, \smallskip \\
\sqrt{\frac{3}{2}},  && |x-1|_p\le \frac{1}{3}, \smallskip \\
0, && |x-2|_p\le \frac{1}{3}, \\
0, && |x|_p\ge 3; \\
\end{array}
\right.
\,
\psi^{(2)}(x)=\left\{
\begin{array}{rll}
-\frac{1}{\sqrt{2}}, && |x|_p\le \frac{1}{3}, \smallskip \\
-\frac{1}{\sqrt{2}}, && |x-1|_p\le \frac{1}{3}, \smallskip \\
\sqrt{2}, && |x-2|_p\le \frac{1}{3}, \\
0, && |x|_p\ge 3. \\
\end{array}
\right.
$$
\end{example}

\bibliographystyle{amsplain}

\begin{thebibliography}{10}


\bibitem{Al-Kh-Sh3}
S.~Albeverio, A.Yu.~Khrennikov, V.M.~Shelkovich,
\textit{Harmonic analysis in the $p$-adic Lizorkin spaces:
fractional operators, pseudo-differential equations,
$p$-adic wavelets, Tauberian theorems},
Journal of Fourier Analysis and Applications,
Vol. 12, Issue 4  (2006), 393--425.

\bibitem{Al-Kh-Sh5}
S.~Albeverio, A.Yu.~Khrennikov, V.M.~Shelkovich,
\textit{$p$-Adic semi-linear evolutionary pseudo-differential equations
in the Lizorkin space},
Dokl. Ross. Akad. Nauk, {\bf 415}, no.~3, (2007), 295--299.
English transl. in Russian Doklady Mathematics,
{\bf 76}, no.~1, (2007), 539--543.

\bibitem{Ben-Ben}
J.J.~Benedetto, and  R.L.~Benedetto,
\textit{A wavelet theory for local fields and related groups},
The Journal of Geometric Analysis \textbf{3} (2004), 423--456.

\bibitem{Kh2}
A.~Khrennikov,
\textit{Non-archimedean analysis: quantum paradoxes, dynamical systems
and biological models}.
Kluwer Academic Publ., Dordrecht, 1997.

\bibitem{Kh-Koz2}
A.Yu.~Khrennikov, and S.V.~Kozyrev,
\textit{Pseudodifferential operators on ultrametric spaces and ultrametric
wavelets},
Izvestia Akademii Nauk, Seria Math.
\textbf{69} no.~5 (2005), 133--148.

\bibitem{Kh-Sh1}
A.Yu.~Khrennikov, V.M.~Shelkovich,
\textit{$p$-Adic multidimensional wavelets and their application to $p$-adic
pseudo-differential operators}. http://arxiv.org/abs/math-ph/0612049

\bibitem{Kh-Sh2}
A.Yu.~Khrennikov, V.M.~Shelkovich,
\textit{Non-Haar $p$-adic wavelets and pseudo-differential operators},
Dokl. Ross. Akad. Nauk, {\bf 418}, no.~2, (2008).
English transl. in Russian Doklady Mathematics, (2008).

\bibitem{Koch3}
A.N.~Kochubei,
\textit{Pseudo-differential equations and stochastics over
non-archimedean fields},
Marcel Dekker. Inc. New York, Basel, 2001.

\bibitem{Koz0}
S.V.~Kozyrev,
\textit{Wavelet analysis as a $p$-adic spectral analysis},
Izvestia Akademii Nauk, Seria Math.
\textbf{66} no.~2 (2002), 149--158.

\bibitem{Koz2}
S.V.~Kozyrev,
\textit{$p$-Adic pseudodifferential operators and $p$-adic wavelets},
Theor. Math. Physics \textbf{138}, no.~3 (2004), 1--42.

\bibitem{NPS}
I.~Novikov, V.~Protassov, and M.~Skopina,
\textit{Wavelet Theory}. Moscow: Fizmatlit, 2005.

\bibitem{S-Sk-1}
V.~M.~Shelkovich, M.~Skopina
\textit{$p$-Adic Haar multiresolution analysis and pseudo-differential operators}.
http://arxiv.org/abs/0705.2294

\bibitem{Vl-V-Z}
V.S.~Vladimirov, I.V.~Volovich and E.I.~Zelenov,
\textit{$p$-Adic analysis and mathematical physics}.
World Scientific, Singapore, 1994.


\end{thebibliography}

\end{document}